\documentclass[12pt]{article}

\usepackage{mathpazo}   
\usepackage{comment}
\usepackage[title]{appendix}
\usepackage{wrapfig}
\usepackage{subcaption}
\usepackage{tikz}
\usepackage{pgfplots}
\pgfplotsset{compat=1.18}
\usetikzlibrary{patterns, pgfplots.fillbetween}
\graphicspath{{.}}
\usepackage[font=footnotesize]{caption}
\usepackage{nameref}
\usepackage{mathtools}
\usepackage{empheq}
\usepackage{comment}
\usepackage[shortlabels,inline]{enumitem}
\setlist[enumerate]{nosep}
\usepackage[colorlinks=true,
linkcolor=refkey,
urlcolor=lblue,
citecolor=red]{hyperref}
\usepackage[doc,wmm,hhb]{optional}
\usepackage{xcolor}

\usepackage{float}
\usepackage{soul}
\usepackage{graphicx}
\definecolor{labelkey}{rgb}{0,0.08,0.45}
\definecolor{refkey}{rgb}{0,0.6,0.0}
\definecolor{Brown}{rgb}{0.45,0.0,0.05}
\definecolor{lime}{rgb}{0.00,0.8,0.0}
\definecolor{lblue}{rgb}{0.5,0.5,0.99}
\definecolor{OliveGreen}{rgb}{0,0.6,0}
\definecolor{tyrianpurple}{rgb}{0.4, 0.01, 0.24}

\colorlet{hlcyan}{cyan!30}

\usepackage{stmaryrd}
\usepackage{amssymb}

\makeatletter
\def\namedlabel#1#2{\begingroup
   \def\@currentlabel{#2}%
   \label{#1}\endgroup
}
\makeatother

\newcommand{\seppthree}{\setlength{\itemsep}{-3pt}}

\usepackage[margin=0.92in,footskip=0.25in]{geometry}
\newcommand{\nnn}{\ensuremath{{n\in{\mathbb N}}}}

\newcommand{\thalb}{\ensuremath{\tfrac{1}{2}}}

\newcommand{\To}{\ensuremath{\rightrightarrows}}

\newcommand{\fenv}[1]%
{\ensuremath{\,\overrightarrow{\operatorname{env}}_{#1}}}
\newcommand{\benv}[1]%
{\ensuremath{\,\overleftarrow{\operatorname{env}}_{#1}}}

\newcommand{\scal}[2]{\left\langle{#1},{#2}  \right\rangle}

\newcommand{\cerouno}{\ensuremath{\left]0,1\right[}}
\newcommand{\RR}{\ensuremath{\mathbb R}}

\newcommand{\ran}{\ensuremath{{\operatorname{ran}}\,}}

\newcommand{\Fix}{\ensuremath{\operatorname{Fix}}}
\newcommand{\Id}{\ensuremath{\operatorname{Id}}}

\newcommand{\pinf}{\ensuremath{+\infty}}

\providecommand{\fejer}{Fej\'{e}r}
{\begin{list}{}{%
\settowidth{\labelwidth}{\textrm{#1~}}%
\setlength{\leftmargin}{\labelwidth+\labelsep}}}
{\end{list}}
\usepackage{amsthm}
\usepackage{aliascnt}
\makeatletter
\def\th@plain{%
	\thm@notefont{}
	\itshape 
}
\def\th@definition{%
	\thm@notefont{}
	\normalfont 
}
\makeatother
\usepackage[capitalize,nameinlink]{cleveref}
\crefname{equation}{}{equations}

\crefname{item}{}{items}
\crefname{enumi}{}{}
\newtheorem{theorem}{Theorem}[section]

\newaliascnt{lemma}{theorem}

\aliascntresetthe{lemma}
\crefname{lemma}{Lemma}{Lemmas}
\Crefname{lemma}{Lemma}{Lemmas}

\newaliascnt{lem}{theorem}

\aliascntresetthe{lem}
\crefname{lem}{Lemma}{Lemmas}
\Crefname{lem}{Lemma}{Lemmas}

\newaliascnt{corollary}{theorem}

\aliascntresetthe{corollary}
\crefname{corollary}{Corollary}{Corollaries}
\Crefname{corollary}{Corollary}{Corollaries}

\newaliascnt{cor}{theorem}

\aliascntresetthe{cor}
\crefname{cor}{Corollary}{Corollaries}
\Crefname{cor}{Corollary}{Corollaries}

\newaliascnt{proposition}{theorem}
\newtheorem{proposition}[proposition]{Proposition}
\aliascntresetthe{proposition}
\crefname{proposition}{Proposition}{Propositions}
\Crefname{proposition}{Proposition}{Propositions}

\newaliascnt{prop}{theorem}

\aliascntresetthe{prop}
\crefname{prop}{Proposition}{Propositions}
\Crefname{prop}{Proposition}{Propositions}

\newaliascnt{definition}{theorem}

\aliascntresetthe{definition}
\crefname{definition}{Definition}{Definitions}
\Crefname{definition}{Definition}{Definitions}

\newaliascnt{defn}{theorem}

\aliascntresetthe{defn}
\crefname{defn}{Definition}{Definitions}
\Crefname{defn}{Definition}{Definitions}

\newaliascnt{thm}{theorem}

\aliascntresetthe{thm}
\crefname{thm}{Theorem}{Theorems}
\Crefname{thm}{Theorem}{Theorems}

\newaliascnt{example}{theorem}

\aliascntresetthe{example}
\crefname{example}{Example}{Examples}
\Crefname{example}{Example}{Examples}

\newaliascnt{ex}{theorem}
\newtheorem{ex}[ex]{Example}
\aliascntresetthe{ex}
\crefname{ex}{Example}{Examples}
\Crefname{ex}{Example}{Examples}

\newaliascnt{fact}{theorem}
\newtheorem{fact}[fact]{Fact}
\aliascntresetthe{fact}
\crefname{fact}{Fact}{Facts}
\Crefname{fact}{Fact}{Facts}

\newaliascnt{remark}{theorem}
\newtheorem{remark}[remark]{Remark}
\aliascntresetthe{remark}
\crefname{remark}{Remark}{Remarks}
\Crefname{remark}{Remark}{Remarks}

\newaliascnt{rem}{theorem}

\aliascntresetthe{rem}
\crefname{rem}{Remark}{Remarks}
\Crefname{rem}{Remark}{Remarks}

\crefname{chapter}{Appendix}{chapters}

\providecommand{\RR}{\mathbb{R}}

\providecommand{\ran}{\operatorname{ran}}

\providecommand{\Id}{\operatorname{{ Id}}}

\providecommand{\To}{\rightrightarrows}

\providecommand{\ran}{\operatorname{ran}}

\providecommand{\Id}{\operatorname{Id}}

\newcommand{\cran}{\ensuremath{\overline{\operatorname{ran}}\,}}

\providecommand{\RR}{\mathbb{R}}

\definecolor{myblue}{rgb}{0.9,0.9,0.98}

  \newcommand*\mybluebox[1]{%
    \colorbox{myblue}{\hspace{1em}#1\hspace{1em}}}

\allowdisplaybreaks 

\usepackage{relsize}

\begin{document}

\setlength{\abovedisplayskip}{8pt}
\setlength{\belowdisplayskip}{8pt}	
\newsavebox\myboxA
\newsavebox\myboxB
\newlength\mylenA

\newcommand*\xoverline[2][0.75]{%
    \sbox{\myboxA}{$#2$}%
    \setbox\myboxB\null
    \ht\myboxB=\ht\myboxA%
    \dp\myboxB=\dp\myboxA%
    \wd\myboxB=#1\wd\myboxA
    \sbox\myboxB{$\overline{\copy\myboxB}$}
    \setlength\mylenA{\the\wd\myboxA}
    \addtolength\mylenA{-\the\wd\myboxB}%
    \ifdim\wd\myboxB<\wd\myboxA%
       \rlap{\hskip 0.5\mylenA\usebox\myboxB}{\usebox\myboxA}%
    \else
        \hskip -0.5\mylenA\rlap{\usebox\myboxA}{\hskip 0.5\mylenA\usebox\myboxB}%
    \fi}
\makeatother

\makeatletter
\renewcommand*\env@matrix[1][\arraystretch]{%
  \edef\arraystretch{#1}%
  \hskip -\arraycolsep
  \let\@ifnextchar\new@ifnextchar
  \array{*\c@MaxMatrixCols c}}
\makeatother

\providecommand{\wbar}{\xoverline[0.9]{w}}
\providecommand{\ubar}{\xoverline{u}}

\newcommand{\nn}[1]{\ensuremath{\textstyle\mathsmaller{({#1})}}}
\newcommand{\crefpart}[2]{%
  \hyperref[#2]{\namecref{#1}~\labelcref*{#1}~\ref*{#2}}%
}
\newcommand\bigzero{\makebox(0,0){\text{\LARGE0}}}
	

%

\author{
Sedi Bartz\thanks{
Mathematics and Statistics, UMass Lowell, MA 01854, USA. E-mail:
\texttt{sedi\_bartz@uml.edu}.},~
Heinz H.\ Bauschke\thanks{
Mathematics, University
of British Columbia,
Kelowna, B.C.\ V1V~1V7, Canada. E-mail:
\texttt{heinz.bauschke@ubc.ca}.}~~~and~
Yuan Gao\thanks{
Mathematics, University
of British Columbia,
Kelowna, B.C.\ V1V~1V7, Canada. E-mail:
\texttt{y.gao@ubc.ca}.}
}

\title{\textsf{ 
Baillon-Bruck-Reich revisited: 
divergent-series parameters and strong convergence in the linear case
}
}

\date{March 22, 2026}

\maketitle

\begin{center}
\emph{
To Simeon Reich, on the occasion of his 80th birthday,\\
with deep thanks for decades of inspiration and mentorship. 
}
\end{center}

\begin{abstract}
The Krasnosel'ski\u\i-Mann iteration is an important algorithm in optimization and variational analysis for finding
fixed points of nonexpansive mappings. In the general case, it 
produces a sequence converging \emph{weakly} to a fixed point provided 
the parameter sequence satisfies a divergent-series condition. 

In this paper, we show that \emph{strong} convergence holds provided 
the underlying nonexpansive mapping is \emph{linear}. This improves 
on a celebrated result by Baillon, Bruck, and Reich from 1978, 
where the parameter sequence was assumed to be constant as well as on recent work where 
the parameters were bounded away from $0$ and $1$. 
\end{abstract}


{ 


\section{Introduction}

Throughout this note, 
\begin{empheq}[box=\mybluebox]{equation}
\text{$X$ is a real Hilbert space,}
\end{empheq}
with inner product $\scal{\cdot}{\cdot}$ and induced norm $\|\cdot\|$. 
We also assume that 
\begin{empheq}[box=\mybluebox]{equation}
\label{e:T}
\text{$T\colon X\to X$ is nonexpansive, with $F := \Fix T\neq\varnothing$.}
\end{empheq}
When $\lambda\in\RR$, we set
\begin{empheq}[box=\mybluebox]{equation}
T_\lambda := (1-\lambda)\Id + \lambda T,
\end{empheq}
where $\Id\colon X\to X\colon x\mapsto x$ is the identity operator. 

Many problems in optimization and variational analysis 
can be reduced to finding a point in $F$ 
(see, e.g., \cite{BC2017}, \cite{Condat}, \cite{Ryubook}, 
and the references therein). 
A popular algorithmic approach is to employ the
\emph{Krasnosel'ski\u\i-Mann iteration} for which we now state a basic 
convergence result:

\begin{fact}[Reich]
\label{f:Reich}
Let $(\lambda_n)_\nnn$ be in $[0,1]$ and assume that 
$\sum_\nnn (1-\lambda_n)\lambda_n=\pinf$. 
Let $x_0\in X$.
Then the sequence generated by 
\begin{equation}
(\forall\nnn)\quad
x_{n+1} := T_{\lambda_n}x_n
\end{equation}
converges weakly to a point in $F$. 
Moreover, 
\begin{equation}
\label{e:Reich}
x_n-Tx_n \to 0,
\end{equation}
and $(x_n)_\nnn$ is \fejer\ monotone with respect to $F$. 
\end{fact}
\begin{proof}
This is \cite[Theorem~2]{Reich}, which holds in significantly more general settings. (See also \cite[Theorem~5.15]{BC2017}.)
\end{proof}

We refer the reader also to the recent book \cite{KMbook}, 
dedicated entirely to 
Krasnosel'ski\u\i-Mann iterations. 
When $T$ is \emph{linear}, the Krasnosel'ski\u\i-Mann iteration 
produces a \emph{strongly} convergent sequence provided that the 
parameter sequence $(\lambda_n)_\nnn$ is \emph{constant}: 

\begin{fact}[Baillon-Bruck-Reich] 
\label{f:BBR}
Suppose that $T$ is \emph{linear} and that $\lambda\in\cerouno$. 
Let $x_0\in X$, and generate sequence $(x_n)_\nnn$ in $X$ via 
\begin{equation}
(\forall\nnn)\quad
x_{n+1} := T_{\lambda}x_n
\end{equation}
Then
\begin{equation}
x_n \to P_{F}x_0. 
\end{equation}
\end{fact}
\begin{proof}
This is \cite[Example~5.29]{BC2017}; however, 
the main ideas of the proof are in
\cite{BBR} and \cite{BR}. 
\end{proof}

Recently, \cref{f:BBR} was generalized in \cite{BG} to deal with 
a parameter sequence $(\lambda_n)_\nnn$ satisfying 
$0<\inf_\nnn \lambda_n \leq \sup_\nnn \lambda_n < 1$. 

\emph{The goal of this paper is to extend the Baillon-Bruck-Reich
result (\cref{f:BBR}) to the general parameter sequence 
$(\lambda_n)_\nnn$ satisfying the divergent-series condition $\sum_\nnn (1-\lambda_n)\lambda_n=\infty$.} 

The remainder of the paper is organized as follows. 
In \cref{s:main}, we present the main result (\cref{t:main}), its proof, 
and corresponding comments. In \cref{sec:last}, we offer 
some concluding remarks.

The notation we employ is fairly standard and follows largely 
\cite{BC2017}.

\section{Main result}
\label{s:main}

\begin{proposition}
\label{p:Fperp}
Recall that $T$ satisfies \cref{e:T}. 
Assume in addition that $T$ is linear. 
Then 
\begin{equation}
\label{e:Fperp}
F^\perp = \cran(\Id-T). 
\end{equation}
\end{proposition}
\begin{proof}
Because $T$ is nonexpansive, we deduce from 
\cite[Example~20.29]{BC2017} that $\Id-T$ is maximally monotone. 
In turn, it follows from 
\cite[Example~20.17]{BC2017} that 
\begin{equation}
F = \ker(\Id-T) = \ker(\Id-T^*).
\end{equation}
On the other hand, 
because $\Id-T$ is a continuous linear operator, we have 
\begin{equation}
\big(\ker(\Id-T^*)\big)^\perp = \cran(\Id-T).
\end{equation}
Altogether, $F^\perp = \cran(\Id-T)$ and so \cref{e:Fperp} holds. 
%
\end{proof}

\begin{remark}
We are grateful to a reviewer who outlined the following alternative 
proof of \cref{p:Fperp}: 
Let $A\colon X\To X$ be a maximally monotone linear relation. 
Combining \cite[Fact~2.1.(v) and Theorem~3.2]{BWY} yields  
$(\ker A)^\perp = \cran A^* =\cran A$.
Then \cref{p:Fperp} follows when one sets $A:=\Id-T$. 
\end{remark}


We are now ready for the main result.

\begin{theorem}[main result]
\label{t:main}
Recall that $T$ satisfies \cref{e:T}, 
and assume in addition that $T$ is linear. 
Let $(\lambda_n)_\nnn$ be in $[0,1]$ such that 
$\sum_\nnn (1-\lambda_n)\lambda_n=\infty$, 
and let $x_0\in X$.
Then the sequence generated by 
\begin{equation}
(\forall\nnn)\quad
x_{n+1} := T_{\lambda_n}x_n
\end{equation}
converges strongly to $P_Fx_0$. 
\end{theorem}
\begin{proof}
Note that $F \subseteq \Fix T_{\lambda_k}$. 
Set $f_0 := P_Fx_0$ and $g_0 := P_{F^\perp}x_0$ so that 
\begin{equation}
x_0 = f_0+g_0. 
\end{equation}
Clearly,
\begin{align}
x_{n+1} 
&= 
T_{\lambda_n}\cdots T_{\lambda_0}x_0
= 
T_{\lambda_n}\cdots T_{\lambda_0}f_0
+ 
T_{\lambda_n}\cdots T_{\lambda_0}g_0
= 
f_0 + 
T_{\lambda_n}\cdots T_{\lambda_0}g_0 
= 
P_Fx_0 + 
T_{\lambda_n}\cdots T_{\lambda_0}g_0. 
\end{align}
It thus suffices to show that 
\begin{equation}
\label{e:goal}
T_{\lambda_n}\cdots T_{\lambda_0}g_0 \to 0.
\end{equation}
Because $0\in F\subseteq \Fix T_{\lambda_k}$, we have 
that $(\|T_{\lambda_n}\cdots T_{\lambda_0}g_0\|)_\nnn$ is 
\emph{decreasing} and 
so 
\begin{equation}
\label{e:ell}
\ell := \lim_{n\to\infty} \|T_{\lambda_n}\cdots T_{\lambda_0}g_0\|
= \inf_{\nnn} \|T_{\lambda_n}\cdots T_{\lambda_0}g_0\|
\end{equation}
exists.
Next, \cref{e:Fperp} implies 
that 
$g_0 \in F^\perp = \cran(\Id-T)$. 
Let $\varepsilon>0$ and get 
$v_\varepsilon\in X$ such that 
$g_\varepsilon := (\Id-T)v_\varepsilon \in \ran(\Id-T)$ and 
$\|g_0-g_\varepsilon\|\leq\varepsilon$. 
Now \cref{e:Reich} from \cref{f:Reich} yields 
\begin{equation}
\label{e:to0}
T_{\lambda_n}\cdots T_{\lambda_0}g_\varepsilon
= 
T_{\lambda_n}\cdots T_{\lambda_0}(\Id-T)v_\varepsilon
= 
T_{\lambda_n}\cdots T_{\lambda_0}v_\varepsilon 
-TT_{\lambda_n}\cdots T_{\lambda_0}v_\varepsilon
\to 
0, 
\end{equation}
where we also used the fact that $T$ commutes with each $T_{\lambda_i}$.
Thus 
\begin{align}
\ell
&\leq 
\|T_{\lambda_n}\cdots T_{\lambda_0}g_0\|
\tag{by \cref{e:ell}}\\
&\leq 
\|T_{\lambda_n}\cdots T_{\lambda_0}g_0 
- T_{\lambda_n}\cdots T_{\lambda_0}g_\varepsilon\|
+ \|T_{\lambda_n}\cdots T_{\lambda_0}g_\varepsilon\|
\tag{triangle inequality}
\\
&\leq 
\|g_0-g_\varepsilon\| + 
\|T_{\lambda_n}\cdots T_{\lambda_0}g_\varepsilon\|
\tag{each $T_{\lambda_k}$ is nonexpansive}
\\
&\leq 
\varepsilon + 
\|T_{\lambda_n}\cdots T_{\lambda_0}g_\varepsilon\|
\tag{choice of $v_\varepsilon$ and $g_\varepsilon$}
\\
&\to \varepsilon 
\;\;\text{ as $n\to \infty$.}
\tag{by \cref{e:to0}}
\end{align}
Because $\varepsilon$ was chosen arbitrarily, we deduce that 
$\ell=0$, which verifies \cref{e:goal} and we're done. 
\end{proof}

We conclude with some remarks concerning \cref{t:main}. 

\begin{remark}[without linearity, strong convergence fails]
Thanks to an example constructed by 
Genel and Lindenstrauss (see \cite{GL}), 
it is known that the sequence $(x_n)_\nnn$ generated by 
the Krasnosel'ski\u\i-Mann iteration may fail to converge strongly
even when $\lambda_n\equiv\thalb$. 
A posteriori, the underlying nonexpansive mapping cannot be linear. 
\end{remark}

\begin{remark}[without divergent-series parameters, 
convergence to a fixed point fails]
Suppose that $x_0\in X\smallsetminus\{0\}$, and consider 
the \emph{linear} nonexpansive isometry $T = -\Id$. 
Then $F=\{0\}$. 
Assume that the parameter sequence $(\lambda_n)_\nnn$ lies in $\left[0,\thalb\right[$. 
Using mathematical induction, one checks that 
the sequence $(x_n)_\nnn$ generated by the Krasnosel'ski\u\i-Mann 
iteration satisfies
\begin{equation}
x_n = \Big(\prod_{k=0}^{n-1}(1-2\lambda_k)\Big) x_0; 
\end{equation}
moreover, $x_n \to 0$ 
$\Leftrightarrow$
$\sum_\nnn \lambda_n = \infty$. 
Therefore, if $\sum_\nnn (1-\lambda_n)\lambda_n <\infty$, then 
$(x_n)_\nnn$ does \emph{not} converge to a point in $F$. 
\end{remark}

\begin{remark}[extension to the affine case]
If $T\colon X\to X$ is assumed to be \emph{affine}, 
still with $F = \Fix T\neq\varnothing$, then 
\cref{t:main} holds as well. 
The (standard translation) argument is outlined in \cite[Section~3.2]{BG}. 
\end{remark}

\begin{remark}[for averaged $T$, one may increase the range of the parameter sequence]
If it is a priori known that $T$ is \emph{averaged}, then 
one may enlarge the range of parameter sequence to go beyond $1$. 
(See \cite[Section~3.1]{BG} for how this is done.)
\end{remark}


\section{Conclusions}

\label{sec:last}
A well-known result by Reich on Krasnosel'ski\u\i-Mann iterations yields weak convergence to a fixed point 
of a nonexpansive mapping $T$ under a \emph{general (divergent-series)} assumption on the parameters. On the other hand, when $T$ is linear, 
work by Baillon, Bruck and Reich implies that the convergence is \emph{strong} but the relaxation parameters needed to be constant. In this paper, we extend 
the Baillon-Bruck-Reich result from constant to general parameters in the linear case. The limit is explicitly identified, and 
some limiting examples and extensions are discussed. It would be interesting to know if this result remains true if 
$T$ is only assumed to be \emph{odd} instead of linear as was the case in
\cite{BBR}.


\section*{Acknowledgments}
The authors thank the reviewers and the editor for 
constructive comments which significantly improved the manuscript. The research of SB was partially supported by a collaboration grant for mathematicians of the Simons Foundation. The research of HHB was partially supported by a Discovery Grant 
of the Natural Sciences and Engineering Research Council of
Canada.



\begin{thebibliography}{999}
\seppthree

\bibitem{BBR}
Baillon, J.B., Bruck, R.E., Reich, S.: 
On the asymptotic behavior of nonexpansive mappings
and semigroups in Banach spaces.
{Houston J.\ Math.}~\textbf{4}, 1--9 (1978)

\bibitem{BC2017}
Bauschke, H.H., Combettes, P.L.: 
{Convex Analysis and Monotone Operator Theory in Hilbert Spaces}.
2nd edition, Springer (2017)

\bibitem{BG}
Bauschke, H.H., Gao, Y.:
On a result by Baillon, Bruck, and Reich. 
Vietnam J.\ Math.~\textbf{53}, 893--900 (2025) 

\bibitem{BWY}
Bauschke, H.H., Wang, X., Yao, L.:
Monotone linear relations: maximality and Fitzpatrick functions. 
J.\ Convex Anal.~\textbf{16}, 673--686 (2009)

\bibitem{BR}
Bruck, R.E., Reich, S.: 
Nonexpansive mappings and resolvents of accretive operators 
in Banach spaces.
Houston J.\ Math.~\textbf{3}, 459--470 (1977)

\bibitem{Condat}
Condat, L., Kitahara, D., Contreras, A., Hirabayashi, A.:
Proximal splitting algorithms for convex optimization: a tour
of recent advances, with new twists. 
SIAM Rev.~\textbf{65}, 375--435 (2023)

\bibitem{KMbook}
Dong, Q.-L., Cho, Y.J.,
He, S., Pardalos, P.M., 
Rassias, T.M.:
The Krasnosel'ski\u\i-Mann Iterative Method. 
Springer (2022)

\bibitem{GL}
Genel, A., Lindenstrauss, J.:
An example concerning fixed points. 
Israel J.\ Math.~\textbf{22}, 81--86 (1975) 


\bibitem{Reich}
Reich, S.:
Weak convergence theorems for nonexpansive mappings in Banach spaces. 
J.\ Math.\ Anal.\ Appl.~\textbf{67}, 274--276 (1979) 

\bibitem{Ryubook}
Ryu, E.K., Yin, W.:
Large-Scale Convex Optimization. Cambridge University Press 
(2023)


\end{thebibliography}
\end{document}